\documentclass[12pt]{article}
\usepackage{amsmath,amssymb,amsthm}
\usepackage{cases}
\usepackage{amsfonts}
\usepackage{color,xcolor,cite}
\usepackage[integrals]{wasysym}
\usepackage[left=2.0cm,right=2.0cm,top=2.0cm,bottom=2.0cm]{geometry}
\usepackage[colorlinks,citecolor=blue,urlcolor=blue]{hyperref}

\numberwithin{equation}{section}

\newtheorem{theorem}{Theorem}[section]

\newtheorem{lemma}{Lemma}[section]
\newtheorem{proposition}{Proposition}[section]

\newtheorem{definition}{Definition}[section]
\newtheorem{remark}{Remark}[section]

\newcommand{\bal}{\begin{align}}
\newcommand{\bbal}{\begin{align*}}
\newcommand{\beq}{\begin{equation}}
\newcommand{\eeq}{\end{equation}}
\newcommand{\bca}{\begin{cases}}
\newcommand{\eca}{\end{cases}}

\newcommand{\pa}{\partial}
\newcommand{\fr}{\frac}
\newcommand{\na}{\nabla}
\newcommand{\De}{\Delta}

\newcommand{\cd}{\cdot}
\newcommand{\ep}{\varepsilon}
\newcommand{\dd}{\mathrm{d}}

\newcommand{\R}{\mathbb{R}}
\newcommand{\vv}{\mathbf{v}}

\newcommand{\bi}{\Big}

\begin{document}
\title{Ill-posedness for the higher dimensional Camassa-Holm equations in Besov spaces }

\author{Min Li$^{1}$ and Yingying Guo$^{2,}$\footnote{E-mail: limin@jxufe.edu.cn; guoyy35@fosu.edu.cn; } \\
\small $^1$ School of information Technology, Jiangxi University of Finance and Economics, Nanchang, 330032, China\\
\small $^2$ Department of Mathematics, Foshan University, Foshan, Guangdong 528000, China}

\date{}

\maketitle\noindent{\hrulefill}

{\bf Abstract:} In the paper, by constructing a initial data $u_{0}\in B^{\sigma}_{p,\infty}$ with $\sigma-2>\max\{1+\frac 1 p, \frac 3 2\}$, we prove that the corresponding solution to the higher dimensional Camassa-Holm equations starting from $u_{0}$ is discontinuous at $t=0$ in the norm of $B^{\sigma}_{p,\infty}$, which implies that the ill-posedness for the higher dimensional Camassa-Holm equations in $B^{\sigma}_{p,\infty}$.


\vskip0mm\noindent{\hrulefill}

\section{Introduction}\label{sec1}
In this paper, we consider the following initial value problem for the higher dimensional Camassa-Holm equations:
\begin{equation}\label{E-P}
\begin{cases}
\partial_tm+u\cdot \nabla m+\nabla u^T\cd m+(\mathrm{div} u)m=0, \qquad &(t,x)\in \R^+\times \R^d,\\
m=(1-\De)u,\qquad &(t,x)\in \R^+\times \R^d,\\
u(0,x)=u_0,\qquad &x\in \R^d.
\end{cases}
\end{equation}

According to \cite{yy}, we can transform Eq. \eqref{E-P} into the following form of transport equations:
\begin{align}\label{E-P1}
\partial_tu+u\cdot \nabla u= Q(u,u)+R(u,u),
\end{align}
where
\begin{align*}
&Q(u,u)=-(I-\De)^{-1}\mathrm{div}\Big(\nabla u\nabla u+\nabla u\nabla u^T-\nabla u^T\nabla u-\nabla u(\mathrm{div} u)+\frac12\mathbf{I}|\nabla u|^{2}\Big),
\\&R(u,u)=-(I-\De)^{-1}\Big(u(\mathrm{div} u)+ u\cd (\na u)^T\Big).
\end{align*}

Eq. \eqref{E-P} was investigated as Euler-Poincar\'{e} equations describes geodesic motion on the diffeomorphism group with respect to the kinetic energy norm in \cite{hs}, which can also be viewed as higher dimensional generalization of the following classical one dimensional Camassa-holm equation (CH):
\begin{align*}
m_{t}+um_{x}+2u_{x}m=0,\ m=u-u_{xx}.
\end{align*}
The CH equation is completely integrable \cite{c3,cgi} and has a bi-Hamiltonian structure \cite{c1,ff}. It also has the solitary waves and peak solitons \cite{cs,cs1}. It is worth mentioning that the peakons show the characteristic for the traveling waves of greatest height and arise as solutions to the free-boundary problem for the incompressible Euler equations over a flat bed, see \cite{c5,c6,ce7,t}. The local well-posedness, global strong solutions, blow-up strong solutions of the CH equations were studied in \cite{c2,ce4,ce2,ce3,d1,lio,rb}. The global weak solutions, global conservative solutions and dissipative solutions also have been investigated in \cite{bc1,bc2,cmo,hr1,xz1}. For the continuity of the solutions map of the CH equations with respect to the initial data, it was only proved in the spaces $C([0,T];B^{s'}_{p,r}(\R))$ for any $s'<s$ with $s>\max\{\frac{3}{2},1+\frac{1}{p}\}$ by many authors. Moreover, Li and Yin in \cite{liy} proved that the index of the continuous dependence for the solutions to the Camassa-Holm type equations in $B^{s}_{p,r}(\R) \big(s>\max\{\frac{3}{2},1+\frac{1}{p}\}\big)$ can up to $s$, which improved many authors' results, especially the Danchin's results in \cite{d1,d2}. Recently, Guo et al. \cite{glmy} obtained the local ill-posedness for a class of shallow water wave equations (such as, the CH, DP, Novikov equations and etc.) in critical Sobolev space $H^{\frac 3 2}(\R)$ and even in Besov space $B^{1+\frac 1 p}_{p,r}(\R)$ with $p\in[1,+\infty],\ r\in(1,+\infty]$. More recently, by use of the compactness argument and Lagrangian coordinate transformatiorather rather than the usual techniques used in \cite{liy}, Ye et al. \cite{yyg}  proved the CH equation is locally well-posed and continuous dependence in Besov spaces $B^{1+\frac 1 p}_{p,1}(\R)$ with $p\in[1,+\infty)$, which implied $B^{1+\frac 1 p}_{p,1}(\R)$ is the critical Besov spaces and the index $\frac{3}{2}$ is not necessary for the Camassa-Holm type equations. Further, the non-uniform continuity of the CH equation has been investigated in many papers, see \cite{hmp,hk,hkm,lyz1,lwyz}.

Eqs. \eqref{E-P} has numerous remarkable features and has been studies by many authors. Chae and Liu \cite{cli} eatablished the local existence of weak solution in $W^{2,p}(\R^d),\ p>d$ and local existence of unique classical solutions in $H^{s}(\R^d),\ s>\frac d 2+3$ for \eqref{E-P}. Then, Li, Yu and Zhai \cite{lyzz} proved that the solutions to \eqref{E-P} with a large class of smooth initial data blows up in finite time or exists globally in time, which settled an open problem raised by Chae and Liu \cite{cli}. Taking advantage of the Littlewood-Paley decomposition theory, Yan and Yin \cite{yy} further discussed the local existence and uniqueness of the solution to \eqref{E-P} in Besov spaces $B^{s}_{p,r}(\R^d)$ with $s>\max\{1+\frac d p, \frac 3 2\}$ and $s=1+\frac d p,1\leq p\leq 2d,r=1$. Recently, Li, Dai and Zhu \cite{ldz} shown that the corresponding solution to \eqref{E-P} is not uniformly constinuous dependence for that the initial data in $H^{s}(\R^d),s>1+\frac d 2$. Also, Li, Dai and Li in \cite{ldl} have shown that the data-to-solution map for \eqref{E-P} is not uniformly continuous dependence in Besov spaces $B^s_{p,r}(\R^d),s>\max\{1+\frac d2,\frac32\}$.  For more results of higher dimensional Camassa-Holm equations, see \cite{luoy,zyl}.

However, the continuity of the data-to-solution map for the higher dimensional Camassa-Holm equations in Besov spaces $B^{s}_{p,\infty}(\R^d),\ s>\max\{1+\frac d p,\frac 3 2\},\ 1\leq p\leq+\infty$ has not been solved yet. In this paper, we will pay our attention to studying the ill-posedness for the higher dimensional Camassa-Holm equations in Besov spaces. The key skill is to construct a initial data.

Now let us state our main result of this paper.
\begin{theorem}\label{th1}
Let $d\geq 2$ and $\sigma>2+\max\big\{1+\frac d p,\frac 3 2\big\}$ with $1\leq p\leq \infty$. There exists a $u_0\in B^\sigma_{p,\infty}(\R^{d})$ and a positive constant $\ep_0$ such that the data-to-solution map $u_0\mapsto u(t)$ of the Cauchy problem \eqref{E-P} satisfies
\begin{align*}
\limsup_{t\to0^+}\|u(t)-u_0\|_{B^\sigma_{p,\infty}}\geq \ep_0.
\end{align*}
\end{theorem}
\begin{remark}
Theorem \ref{th1} demonstrates the ill-posedness of the higher dimensional CH equation in $B^\sigma_{p,\infty}$. More precisely, there exists a $u_0\in B^\sigma_{p,\infty}$ such that the  corresponding solution to the higher dimensional CH equation that starts from $u_0$ does not converge back to $u_0$ in the sense of $B^\sigma_{p,\infty}$-norm as time goes to zero. Our key argument is to construct a initial data $u_0$.
\end{remark}

The remainder of this paper is organized as follows. In Section \ref{sec2}, we list some notations and recall known results. In Section \ref{sec3}, we present the proof of Theorem \ref{th1} by establishing some technical lemmas and propositions.

\section{Preliminaries}\label{sec2}
\subsection{General Notation}\label{sec2.1}
In the following, we denote by $\star$ the convolution.
Given a Banach space $X$, we denote its norm by $\|\cdot\|_{X}$. For $I\subset\R$, we denote by $\mathcal{C}(I;X)$ the set of continuous functions on $I$ with values in $X$. Sometimes we will denote $L^p(0,T;X)$ by $L_T^pX$.
For all $f\in \mathcal{S}'$, the Fourier transform $\mathcal{F}f$ (also denoted by $\widehat{f}$) is defined by
$$
\mathcal{F}f(\xi)=\widehat{f}(\xi)=\int_{\R^{d}}e^{-ix\cdot\xi}f(x)\dd x \quad\text{for any}\; \xi\in\R^{d}.
$$
\subsection{Littlewood-Paley Analysis}\label{sec2.2}
Next, we recall some useful properties about the Littlewood-Paley decomposition and the Besov spaces.
\begin{proposition}[Littlewood-Paley decomposition, See \cite{bcd}] Let $\mathcal{B}:=\{\xi\in\mathbb{R}^{d}:|\xi|\leq \frac 4 3\}$ and $\mathcal{C}:=\{\xi\in\mathbb{R}^{d}:\frac 3 4\leq|\xi|\leq \frac 8 3\}.$
There exist two radial functions $\chi\in C_c^{\infty}(\mathcal{B})$ and $\varphi\in C_c^{\infty}(\mathcal{C})$ both taking values in $[0,1]$ such that
\begin{align*}
&\chi(\xi)+\sum_{j\geq0}\varphi(2^{-j}\xi)=1, \quad \forall \;  \xi\in \R^d,\\
&\frac{1}{2} \leq \chi^{2}(\xi)+\sum_{j \geq 0} \varphi^{2}(2^{-j} \xi) \leq 1,\quad \forall \;  \xi\in \R^d.
\end{align*}
\end{proposition}
\begin{proposition}[Bernstein's inequalities, See \cite{bcd}]\label{bern}
Let $\mathcal{B}$ be a ball and $\mathcal{C}$ be an annulus. A constant $C>0$ exists such that for all $k\in\mathbb{N},\ 1\leq  p\leq q\leq\infty$, and any function $f\in L^{p}(\mathbb{R})$, we have
\begin{align*}
&Supp{\widehat{f}}\subset\lambda\mathcal{B}\Rightarrow\|D^{k}f\|_{L^{q}}=\sup\limits_{|\alpha|=k}\|\partial^{\alpha}f\|_{L^{q}}\leq C^{k+1}\lambda^{k+d(\frac1p-\frac1q)}\|f\|_{L^{p}},\\
&Supp{\widehat{f}}\subset\lambda\mathcal{C}\Rightarrow C^{-k-1}\lambda^{k}\|f\|_{L^{p}}\leq\|D^{k}f\|_{L^{q}}\leq C^{k+1}\lambda^{k}\|f\|_{L^{p}}.
\end{align*}

\end{proposition}

\begin{definition}[See \cite{bcd}]
For any $u\in \mathcal{S'}(\mathbb{R}^{d})$, the Littlewood-Paley dyadic blocks ${\Delta}_j$ are defined as follows
\begin{numcases}{\Delta_ju=}
0, & if $j\leq-2$;\nonumber\\
\chi(D)u=\mathcal{F}^{-1}(\chi \mathcal{F}u), & if $j=-1$;\nonumber\\
\varphi(2^{-j}D)u=\mathcal{F}^{-1}\big(\varphi(2^{-j}\cdot)\mathcal{F}u\big), & if $j\geq0$.\nonumber
\end{numcases}
In the nonhomogeneous case, the following Littlewood-Paley decomposition makes sense:
$$
u=\sum_{j\geq-1}{\Delta}_ju,\quad \forall\;u\in \mathcal{S'}(\mathbb{R}^{d}).
$$
\end{definition}
\begin{definition}[See \cite{bcd}]
Let $s\in\mathbb{R}$ and $1\leq p,r\leq \infty$. The nonhomogeneous Besov space $B^{s}_{p,r}(\R^{d})$ is defined by
\begin{align*}
B^{s}_{p,r}(\R^{d}):=\Big\{f\in \mathcal{S}'(\R^{d}):\;\|f\|_{B^{s}_{p,r}(\mathbb{R}^{d})}<\infty\Big\},
\end{align*}
where
\begin{numcases}{\|f\|_{B^{s}_{p,r}(\mathbb{R}^{d})}=}
\left(\sum_{j\geq-1}2^{sjr}\|\Delta_jf\|^r_{L^p(\mathbb{R}^{d})}\right)^{\fr1r}, &if\quad  $1\leq r<\infty$,\nonumber\\
\sup_{j\geq-1}2^{js}\|\Delta_jf\|_{L^p(\mathbb{R}^{d})}, &if\quad  $r=\infty$.\nonumber
\end{numcases}
For simplicity, we always write $u\in B^{s}_{p,r}(\R^{d})$ and $\nabla u\in B^{s}_{p,r}(\R^{d})$ standing for $u\in (B^{s}_{p,r}(\R^{d}))^{d}$ and $\nabla u\in (B^{s}_{p,r}(\R^{d}))^{d^{2}}$, respectively.
\end{definition}
\begin{remark}\label{re3}
It should be emphasized that the following embedding will be often used implicity:
$$B^{s_{1}}_{p,q}(\R^{d})\hookrightarrow B^{s_{2}}_{p,r}(\R^{d})\ \text{for}\;s_{1}>s_{2}\ \text{or}\  s_{1}=s_{2},\ 1\leq q\leq r\leq\infty.$$
\end{remark}
Finally, we give some important lemmas which will be also often used throughout the paper.
\begin{lemma}[See \cite{bcd,yy}]\label{le1}
Let $(p,r)\in[1, \infty]^2$ and $s>\max\big\{1+\frac d p,\frac32\big\}$. Then
\begin{align*}
&\|uv\|_{B^{s-2}_{p,r}(\R^{d})}\leq C\|u\|_{B^{s-2}_{p,r}(\R^{d})}\|v\|_{B^{s-1}_{p,r}(\R^{d})}.
\end{align*}
Hence, for the terms $Q(u,u),\ Q(v,v),\ R(u,u)$ and $R(v,v)$, we have
\begin{align*}
&\|{Q}(u,u)-{Q}(v,v)\|_{B^{s-1}_{p,r}(\R^{d})}\leq C\|u-v\|_{B^{s-1}_{p,r}(\R^{d})}\big(\|u\|_{B^{s}_{p,r}(\R^{d})}+\|v\|_{B^{s}_{p,r}(\R^{d})}\big),\\
&\|{R}(u,u)-{R}(v,v)\|_{B^{s-1}_{p,r}(\R^{d})}\leq C\|u-v\|_{B^{s-1}_{p,r}(\R^{d})}\big(\|u\|_{B^{s}_{p,r}(\R)}+\|v\|_{B^{s}_{p,r}(\R^{d})}\big).
\end{align*}
\end{lemma}
\begin{lemma}[See \cite{bcd}]\label{le2}
For $(p,r)\in[1, \infty]^2$ and $s>0$, $B^s_{p,r}(\R^{d})\cap L^\infty(\R^{d})$ is an algebra. Moreover, for any $u,v \in B^s_{p,r}(\R^{d})\cap L^\infty(\R^{d})$, we have
\begin{align*}
&\|uv\|_{B^{s}_{p,r}(\R^{d})}\leq C\big(\|u\|_{B^{s}_{p,r}(\R^{d})}\|v\|_{L^\infty(\R^{d})}+\|v\|_{B^{s}_{p,r}(\R^{d})}\|u\|_{L^\infty(\R^{d})}\big).
\end{align*}
\end{lemma}

In the paper, we also need some estimates for the following transport equation:
\begin{align}\label{transport}
\left\{\begin{array}{l}
\partial_{t}f+v\cdot \nabla f=g,\ (t,x)\in\mathbb{R}^{+}\times\mathbb{R}^{d}, \\
f(0,x)=f_0(x),\quad x\in\mathbb{R}^{d}.
\end{array}\right.
\end{align}
\begin{lemma}[See \cite{bcd}]\label{existence}
Let $d\in\mathbb{N}^{+},\ 1\leq p\leq\infty,\ 1\leq r\leq\infty$ and $ \theta> -\min(\frac d {p}, 1-\frac d {p})$. Let $f_0\in B^{\theta}_{p,r}(\R^{d})$, $g\in L^1(0,T;B^{\theta}_{p,r}(\R^{d}))$, $v\in L^\rho(0,T;B^{-M}_{\infty,\infty}(\R^{d}))$ for some $\rho>1$ and $M>0$, and
\begin{align*}
\begin{array}{ll}
\nabla v\in L^1(0,T;B^{\frac d {p}}_{p,\infty}(\R^{d})\cap L^{\infty}(\R^{d})), &\ \text{if}\ \theta<1+\frac d {p}, \\
\nabla v\in L^1(0,T;B^{\theta-1}_{p,r}(\R^{d})), &\ \text{if}\ \theta>1+\frac{d}{p}\ (or\ \theta=1+\frac d {p},\ r=1).
\end{array}	
\end{align*}
Then the problem \eqref{transport} has a unique solution $f$ in
\begin{itemize}
\item [-] the space $\mathcal{C}([0,T];B^{\theta}_{p,r}(\R^{d}))$, if $r<\infty$,
\item [-] the space $\big(\bigcap_{{\theta}'<\theta}\mathcal{C}([0,T];B^{{\theta}'}_{p,\infty}(\R^{d}))\big)\bigcap \mathcal{C}_w([0,T];B^{\theta}_{p,\infty}(\R^{d}))$, if $r=\infty$.
\end{itemize}
\end{lemma}

\begin{lemma}[See \cite{bcd,liy}]\label{priori estimate}
Let $d\in\mathbb{N}^{+},\ 1\leq p,r\leq\infty,\ \theta>-\min(\frac{d}{p},1-\frac{d}{p}).$
There exists a constant $C$ such that for all solutions $f\in L^{\infty}(0,T;B^{\theta}_{p,r}(\R^{d}))$ of \eqref{transport} with initial data $f_0$ in $B^{\theta}_{p,r}(\R^{d})$ and $g$ in $L^1(0,T;B^{\theta}_{p,r}(\R^{d}))$, we have, for a.e. $t\in[0,T]$,	
$$ \|f(t)\|_{B^{\theta}_{p,r}(\R^{d})}\leq \|f_0\|_{B^{\theta}_{p,r}(\R^{d})}+\int_0^t\|g(t')\|_{B^{\theta}_{p,r}(\R^{d})}\dd t'+\int_0^t V'(t')\|f(t')\|_{B^{\theta}_{p,r}(\R^{d})}\dd t' $$
or
$$ \|f(t)\|_{B^{\theta}_{p,r}(\R^{d})}\leq e^{CV(t)}\Big(\|f_0\|_{B^{\theta}_{p,r}(\R^{d})}+\int_0^t e^{-CV(t')}\|g(t')\|_{B^{\theta}_{p,r}(\R^{d})}\dd t'\Big) $$
with
\begin{equation*}
V'(t)=\left\{\begin{array}{ll}
\|\nabla v(t)\|_{B^{\frac d p}_{p,\infty}(\R^{d})\cap L^{\infty}(\R^{d})},\ &\text{if}\ \theta<1+\frac{d}{p}, \\
\|\nabla v(t)\|_{B^{\theta-1}_{p,r}(\R^{d})},\ &\text{if}\ \theta>1+\frac{d}{p}\ (\text{or}\ \theta=1+\frac{d}{p},\ r=1).
\end{array}\right.
\end{equation*}
If $f=v$, then for all $\theta>0$, $V'(t)=\|\nabla v(t)\|_{L^{\infty}(\R^{d})}$.
\end{lemma}

\section{Proof of Theorem \ref{th1}}\label{sec3}
\subsection{Construction of Initial Data}\label{sec3.1}
We need to introduce smooth, radial cut-off functions to localize the frequency region. Precisely,
let $\widehat{\phi}\in \mathcal{C}^\infty_0(\mathbb{R})$ be an even, real-valued and non-negative function on $\R$ and satisfy
\begin{numcases}{\widehat{\phi}(\xi)=}
1,&if $|\xi|\leq \frac{1}{4}$,\nonumber\\
0,&if $|\xi|\geq \frac{1}{2}$.\nonumber
\end{numcases}
In \cite{lyz2}, it has been verified that for $f_n:=\phi(x_1)\cos \big(\frac{17}{12}2^{n}x_1\big)\phi(x_2)\cdot\phi(x_n)$ and $n\geq 2$,
\begin{numcases}{\Delta_j(f_n)=}
f_n, &if $j=n$,\nonumber\\
0, &if $j\neq n$.\label{fn}
\end{numcases}
We can obtain the similar result:
\begin{lemma}\label{le4}
Let $6\leq k,n\in \mathbb{N}^+$. Define the function $g^{k}_{m,n}(x)$ by
$$g^{k}_{m,n}(x):=\phi(x_1)\cos \bi(\frac{17}{12}\big(2^{kn}\pm2^{km}\big)x_1\bi)\phi(x_2)\cdots \phi(x_d)\quad\text{with}\quad 0\leq m\leq n-1.$$
Then we have
\begin{numcases}
{\Delta_j(g^k_{m,n})=}
g^k_{m,n}, &if $j=kn$,\nonumber\\
0, &if $j\neq kn$.\nonumber
\end{numcases}
\end{lemma}
\begin{proof}
The proof is similar to that of in \cite{lyz2}, and here we omit it.
\end{proof}

\begin{lemma}\label{le5}
Define the initial data $u_0(x)$ as
\begin{align*}
u_0(x)=\big(u_0^{1}(x),\cdots,u_0^{d}(x)\big):=\Big(\sum\limits^{\infty}_{n=0}2^{-kn\sigma} f^k_n(x),0,\cdots,0\Big),
\end{align*}ce4
where
$$f^k_n(x):=\phi(x_1)\cos \bi(\frac{17}{12}2^{kn}x_1\bi)\phi(x_2)\cdots \phi(x_d),\quad n\geq 0.$$
Then for any $\sigma\in \big(2+\max\{\frac32,1+\frac d p\},+\infty\big)$ and for some $k$ large enough, we have
\begin{align*}
&|u_{0}|^{2}(x)=\big(u_{0}^{1}(x)\big)^{2},\\
&u_{0}\cd \nabla u_{0}=\Big(\frac 1 2\partial_1 \big((u_{0}^{1})^{2}\big),0,\cdots,0\Big)=\Big(\frac 1 2\partial_1\big(|u_{0}|^{2}\big),0,\cdots,0\Big),\\
&\|u_0\|_{B^{\sigma}_{p,\infty}}\leq C,\\
&\|\De_{kn}\big(|u_0|^2\big)\|_{L^p}\geq c2^{-kn\sigma}.
\end{align*}
\end{lemma}
\begin{proof}
Thanks to the definition of Besov spaces, the support of $\varphi(2^{-j}\cdot)$ and \eqref{fn}, we see
\begin{align*}
\|u_0\|_{B^{\sigma}_{p,\infty}}&=\sup_{j\geq -1}2^{j\sigma}\|\Delta_{j}u_0\|_{L^p}\\
&=||\phi||^{d-1}_{L^{p}}\cdot\sup_{j\geq 0}\bi\|\phi(x_1)\cos \bi(\frac{17}{12}2^{j}x_1\bi)\bi\|_{L^p}\\
&\leq C.
\end{align*}
Notice that the simple fact
$$\cos(\mathbf{A}+\mathbf{B})+\cos(\mathbf{A}-\mathbf{B})=2\cos \mathbf{A} \cos \mathbf{B}$$
and
$$\sum^{\infty}_{n=0}\sum^{\infty}_{m=0,m\neq n}\mathbf{A}_n\mathbf{A}_m=2\sum^{\infty}_{n=0}\sum^{n-1}_{m=0}\mathbf{A}_n\mathbf{A}_m,$$
then direct computations give
\begin{align*}
&|u_0|^2(x)=\big(u_{0}^{1}(x)\big)^{2}\\
=&\fr12\sum^{\infty}_{n=0}2^{-2kn\sigma}\phi^2(x_1)
\phi^2(x_2)\cdots\phi^2(x_d)
+\fr12\sum^{\infty}_{n=0}2^{-2kn\sigma}\phi^2(x_1)
\cos\bi(\frac{17}{12}2^{kn+1}x_1\bi)\phi^2(x_2)\cdots\phi^2(x_d)\\
+&\sum^{\infty}_{n=1}\sum^{n-1}_{m=0}2^{-k(n+m)\sigma}\phi^2(x_1)\Big[
\cos\bi(\frac{17}{12}(2^{kn}-2^{km})x_1\bi)+
\cos\bi(\frac{17}{12}(2^{kn}+2^{km})x_1\bi)\Big]\phi^2(x_2)\cdots\phi^2(x_d).
\end{align*}
Lemma \ref{le4} yields
\begin{align*}
&\De_{kn}\big(|u_0|^2\big)=\De_{kn}\big((u_{0}^{1})^{2}\big)\\
=&2^{-kn\sigma}\phi^2(x_1)\Big[
\cos\bi(\frac{17}{12}(2^{kn}-1)x_1\bi)+
\cos\bi(\frac{17}{12}(2^{kn}+1)x_1\bi)\Big]\phi^2(x_2)\cdots\phi^2(x_d)
\\ &+\sum^{n-1}_{m=1}2^{-k(n+m)\sigma}\phi^2(x)\Big[
\cos\bi(\frac{17}{12}(2^{kn}-2^{km})x_1\bi)+
\cos\bi(\frac{17}{12}(2^{kn}+2^{km})x_1\bi)\Big]
\phi^2(x_2)\cdots\phi^2(x_d)\\&=:\mathrm{I}_1+\mathrm{I}_2,
\end{align*}
where we denote
\begin{align*}
&\mathrm{I}_1:=2\cdot2^{-kn\sigma}\phi^2(x_1)
\cos\bi(\frac{17}{12}2^{kn}x_1\bi)\cos\bi(\frac{17}{12}x_1\bi)\phi^2(x_2)
\cdots\phi^2(x_d),\\
&\mathrm{I}_2:=2\sum^{n-1}_{m=1}2^{-k(n+m)\sigma}\phi^2(x_1)
\cos\bi(\frac{17}{12}2^{kn}x_1\bi)\cos\bi(\frac{17}{12}2^{km}x_1\bi)\phi^2(x_2)
\cdots\phi^2(x_d).
\end{align*}
For the first term $\mathrm{I}_1$, after a simple calculation, we have
\begin{align}\label{l0}
\|\mathrm{I}_1\|_{L^p(\R^d)}&\geq 2^{-kn\sigma}\bi\|\phi^2(x_1)\cos\bi(\frac{17}{12}x_1\bi) \cos\bi(\frac{17}{12}2^{kn}x_1\bi) \bi\|_{L^p(\R)}||\phi||^{2(d-1)}_{L^{2p}}.
\end{align}
From Lemma 3.2 in \cite{lyz2}, we have for some $\delta>0$
\begin{align}\label{l2}
\bi\|\phi^2(x_{1})\cos\bi(\frac{17}{12}x_{1}\bi)\cos \bi(\fr{17}{12}2^{kn}x_{1}\bi)\bi\|_{L^p(\R)}&\geq c\big(p,\delta,\phi(0)\big).
\end{align}
Then we obtain from \eqref{l0} that
\begin{align}\label{l3}
\|\mathrm{I}_1\|_{L^p}&\geq c2^{-kn\sigma}.
\end{align}
For the second term $\mathrm{I}_2$, from a straightforward calculation, we deduce
\begin{align}\label{l4}
\|\mathrm{I}_2\|_{L^p}\leq C\|\phi\|^{2d}_{L^{2p}}\sum^{n-1}_{m=1}2^{-k(n+m)\sigma}\leq C2^{-k(n+1)\sigma}.
\end{align}
\eqref{l3} and \eqref{l4} together yield that
\begin{align*}
\|\De_{kn}\big(|u_0|^2\big)\|_{L^p}\geq(c-C2^{-k\sigma})2^{-kn\sigma}.
\end{align*}
We choose $k\geq6$ such that $c-C2^{-k\sigma}>0$ and then finish the proof of Proposition \ref{pro3.2}.
\end{proof}

\subsection{Error Estimates}\label{sec3.2}
We first recall the following local-in-time exstence of strong solutions to \eqref{E-P} in \cite{yy}.
\begin{lemma}[See \cite{yy}]\label{hlocal}
Let $d\geq 2,1\leq p,r\leq\infty$ and $s>\max\{1+\frac{d}{p},\frac{3}{2}\}$. Assume that $u_{0}\in B^{s}_{p,r}(\mathbb{R}^{d})$, then there exists a time $T=T(s,p,r,\|u_{0}\|_{B^{s}_{p,r}(\mathbb{R}^{d})})>0$ such that \eqref{E-P} has a unique solution $u\in \mathcal{C}([0,T];B^{s}_{p,r}(\mathbb{R}^{d}))$. Moreover, for all $t\in[0,T]$, there holds
\begin{align*}achm
\|u(t)\|_{B^{s}_{p,r}(\mathbb{R}^{d})}\leq C\|u_{0}\|_{B^{s}_{p,r}(\mathbb{R}^{d})}.
\end{align*}
\end{lemma}
\begin{proposition}\label{pro3.1}
Let $s=\sigma-2$ and $u_0\in B^{\sigma}_{p,\infty}$. Assume that $u\in L^\infty_TB^{\sigma}_{p,\infty}$ be the solution to the Cauchy problem \eqref{E-P}, we have
\begin{align*}
&\|u(t)-u_0\|_{B^{s-1}_{p,\infty}}\leq Ct\|u_0\|_{B^{s-1}_{p,\infty}}\|u_0\|_{B^{s}_{p,\infty}},
\\&\|u(t)-u_0\|_{B^{s}_{p,\infty}}\leq Ct\big(\|u_0\|^{2}_{B^{s}_{p,\infty}}+\|u_0\|_{B^{s-1}_{p,\infty}}\|u_0\|_{B^{s+1}_{p,\infty}}\big),
\\&\|u(t)-u_0\|_{B^{s+1}_{p,\infty}}\leq Ct\big(\|u_0\|_{B^{s}_{p,\infty}}\|u_0\|_{B^{s+1}_{p,\infty}}+\|u_0\|_{B^{s-1}_{p,\infty}}\|u_0\|_{B^{s+2}_{p,\infty}}\big).
\end{align*}
\end{proposition}
\begin{proof}
Due to Lemma \ref{hlocal}, we know that there exists a positive time $T=T(s,p,r,\|u_0\|_{B^s_{p,\infty}}(\mathbb{R}^{d}))$ such that
\begin{align}\label{s}
\|u(t)\|_{L^\infty_TB^s_{p,\infty}}\leq C\|u_0\|_{B^s_{p,\infty}}\leq C.
\end{align}
Moreover, for $\gamma>\fr12$, taking advantage of Lemma \ref{priori estimate} and \eqref{s}, we have
\begin{align}\label{s1}
\|u(t)\|_{L^\infty_TB^\gamma_{p,\infty}}\leq C\|u_0\|_{B^\gamma_{p,\infty}}.
\end{align}
By the Mean Value Theorem, we obtain from \eqref{E-P1} and \eqref{s1} that
\begin{align*}
\|u(t)-u_0\|_{B^s_{p,\infty}}
&\leq \int^t_0\|\pa_\tau u\|_{B^s_{p,\infty}} \dd\tau
\\&\leq \int^t_0\|Q(u,u)\|_{B^s_{p,\infty}} \dd\tau+\int^t_0\|R(u,u)\|_{B^s_{p,\infty}} \dd\tau+ \int^t_0\|u\cd\na u\|_{B^s_{p,\infty}} \dd\tau
\\&\leq Ct\big(\|u\|^2_{L_t^\infty B^{s}_{p,\infty}}+\|u\|_{L_t^\infty L^\infty}\|\na u\|_{L_t^\infty B^{s}_{p,\infty}}\big)
\\&\leq Ct\big(\|u\|^2_{L_t^\infty B^{s}_{p,\infty}}+\|u\|_{L_t^\infty B^{s-1}_{p,\infty}}\|u\|_{L_t^\infty B^{s+1}_{p,\infty}}\big)
\\&\leq Ct\big(\|u_0\|^2_{B^{s}_{p,\infty}}+\|u_0\|_{B^{s-1}_{p,\infty}}\|u_0\|_{B^{s+1}_{p,\infty}}\big),
\end{align*}
where we have used that $B_{p, \infty}^{s-1}\hookrightarrow L^\infty$ with $s-1>\max\{\frac{d}{p}, \frac{1}{2}\}$.

Following the same procedure as above, according to Lemmas \ref{le1} and \ref{le2}, we see
\begin{align*}
\|u(t)-u_0\|_{B^{s-1}_{p,\infty}}
&\leq \int^t_0\|\pa_\tau u\|_{B^{s-1}_{p,\infty}} \dd\tau
\\&\leq \int^t_0\|Q(u,u)\|_{B^{s-1}_{p,\infty}} \dd\tau+\int^t_0\|R(u,u)\|_{B^{s-1}_{p,\infty}} \dd\tau+ \int^t_0\|u\cd\na u\|_{B^{s-1}_{p,\infty}} \dd\tau
\\&\leq Ct\|u\|_{L_t^\infty B^{s-1}_{p,\infty}}\|u\|_{L_t^\infty B^{s}_{p,\infty}}
\\&\leq Ct\|u_0\|_{B^{s-1}_{p,\infty}}\|u_0\|_{B^{s}_{p,\infty}}
\end{align*}
and
\begin{align*}
\|u(t)-u_0\|_{B^{s+1}_{p,\infty}}
&\leq \int^t_0\|\pa_\tau u\|_{B^{s+1}_{p,\infty}} \dd\tau
\\&\leq \int^t_0\|Q(u,u)\|_{B^{s+1}_{p,\infty}} \dd\tau+\int^t_0\|R(u,u)\|_{B^{s+1}_{p,\infty}} \dd\tau+ \int^t_0\|u\cd\na u\|_{B^{s+1}_{p,\infty}} \dd\tau
\\&\leq Ct\big(\|u\|_{L_t^\infty B^{s}_{p,\infty}}\|u\|_{L_t^\infty B^{s+1}_{p,\infty}}
+\|u\|_{L_t^\infty B^{s-1}_{p,\infty}}\|u\|_{L_t^\infty B^{s+2}_{p,\infty}}\big)
\\&\leq Ct\big(\|u_0\|_{B^{s}_{p,\infty}}\|u_0\|_{B^{s+1}_{p,\infty}}+\|u_0\|_{B^{s-1}_{p,\infty}}\|u_0\|_{B^{s+2}_{p,\infty}}\big).
\end{align*}
Thus, we finish the proof of Proposition \ref{pro3.1}.
\end{proof}

\begin{proposition}\label{pro3.2}
Let $s=\sigma-2$ and $u_0\in B^{\sigma}_{p,\infty}$. Assume that $u\in L^\infty_TB^{\sigma}_{p,\infty}$ be the solution of the Cauchy problem \eqref{E-P}, we have
\begin{align*}
\|\mathbf{w}(t,u_0)\|_{B^{s}_{p,\infty}}\leq Ct^2\big(\|u_0\|^3_{B^s_{p,\infty}}+\|u_0\|_{B^{s-1}_{p,\infty}}\|u_0\|_{B^{s}_{p,\infty}}\|u_0\|_{B^{s+1}_{p,\infty}}
+\|u_0\|_{B^{s-1}_{p,\infty}}^{2}\|u_0\|_{B^{s+2}_{p,\infty}}\big),
\end{align*}
here and in what follows we denote
\begin{align*}
&\mathbf{w}(t,u_0):=u(t)-u_0-t\mathbf{v}_0,\\
&\mathbf{v}_0:=-u_0\cd\na u_0+{Q}(u_0,u_{0})+{R}(u_0,u_{0}).
\end{align*}
In particular, we obtain
\begin{align*}
\|\mathbf{w}(t,u_0)\|_{B^{\sigma-2}_{p,\infty}}\leq C\big(\|u_0\|_{B^{\sigma}_{p,\infty}}\big)t^{2}.
\end{align*}
\end{proposition}
\begin{proof}
Taking advantage of the Mean Value Theorem and \eqref{E-P1},  and then using Lemma \ref{le1} and Lemma \ref{le2}, we find
\begin{align*}
\|\mathbf{w}(t,u_0)\|_{B^s_{p,\infty}}
\leq& \int^t_0\|\pa_\tau u-\mathbf{v}_0\|_{B^s_{p,\infty}} \dd\tau
\\
\leq&\int^t_0\|{Q}(u,u)-{Q}(u_0,u_{0})\|_{B^s_{p,\infty}}+\|{R}(u,u)-{R}(u_0,u_{0})\|_{B^s_{p,\infty}} \dd\tau\\
&+\int^t_0\|u\cd\na u-u_0\cd\na u_0\|_{B^s_{p,\infty}} \dd\tau
\\
\leq&C\int^t_0\|u(\tau)-u_0\|_{B^s_{p,\infty}} \|u_0\|_{B^s_{p,\infty}}\dd\tau+C\int^t_0\|u(\tau)-u_0\|_{B^{s-1}_{p,\infty}} \|u(\tau)\|_{B^{s+1}_{p,\infty}} \dd\tau
\\
&+C\int^t_0\|u(\tau)-u_0\|_{B^{s+1}_{p,\infty}}\|u_0\|_{B^{s-1}_{p,\infty}}\dd \tau
\\
\leq&Ct^2\big(\|u_0\|^3_{B^s_{p,\infty}}+\|u_0\|_{B^{s-1}_{p,\infty}}\|u_0\|_{B^{s}_{p,\infty}}\|u_0\|_{B^{s+1}_{p,\infty}}+\|u_0\|_{B^{s-1}_{p,\infty}}^{2}\|u_0\|_{B^{s+2}_{p,\infty}}\big),
\end{align*}
where we have used Proposition \ref{pro3.1} in the last step.

Thus, we complete the proof of Proposition \ref{pro3.2}.

Now we present the proof of Theorem \ref{th1}.
\end{proof}
\begin{proof}[\bf Proof of Theorem \ref{th1}:]
Using Proposition \ref{bern}, Proposition \ref{pro3.2} and Lemma \ref{le5}, we get
\begin{align*}
\|u(t)-u_0\|_{B^\sigma_{p,\infty}}
\geq&2^{{kn\sigma}}\big\|\De_{kn}\big(u(t)-u_0\big)\big\|_{L^p}=2^{{kn\sigma}}\big\|\De_{kn}\big(t\vv_0+\mathbf{w}(t,u_0)\big)\big\|_{L^p}\\
\geq& t2^{{kn\sigma}}\|\De_{kn}\big(\mathbf{v}_0\big)\|_{L^p}-2^{{2kn}}2^{{kn(\sigma-2)}}\big\|\De_{kn}\big(\mathbf{w}(t,u_0)\big)\big\|_{L^p}\\
\geq& t2^{{kn}(\sigma+1)}\|\De_{kn}\big(|u_0|^2\big)\|_{L^p}-
Ct(\|u_0\|_{B^{\sigma-1}_{p,\infty}}^{2}+\|\na u_0\|_{B^{\sigma-1}_{p,\infty}}^{2})\\
&-C2^{2{kn}}\|\mathbf{w}(t,u_0)\|_{B^{\sigma-2}_{p,\infty}}
\\
\geq& t2^{{kn}(\sigma+1)}\|\De_{kn}\big(|u_0|^2\big)\|_{L^p}-
Ct-C2^{2{kn}}t^2\\
\geq& ct2^{{kn}}-Ct-C2^{2{kn}}t^2.
\end{align*}
Then, for $k\geq 6$, taking $n>N$ large  enough such that $c2^{{kn}}\geq 2C$, we deduce that
\begin{align*}
\|u(t)-u_0\|_{B^\sigma_{p,\infty}}\geq ct2^{{kn}}-C2^{2{kn}}t^2.
\end{align*}
Thus, choosing $t2^{kn}\approx\ep$ with small $\ep$, we finally conclude that
\begin{align*}
\|u(t)-u_0\|_{B^\sigma_{p,\infty}}\geq c\ep-C\ep^2\geq c_1\ep.
\end{align*}
This proves Theorem \ref{th1}.
\end{proof}

\vspace*{1em}
\noindent\textbf{Acknowledgements.}
Y. Guo was supported by the Guangdong Basic and Applied Basic Research Foundation (No. 2020A1515111092) and Research Fund of Guangdong-Hong Kong-Macao Joint Laboratory for Intelligent Micro-Nano Optoelectronic Technology (No. 2020B1212030010).

\vspace*{1em}
\noindent\textbf{Conflict of interest}
The authors declare that they have no conflict of interest.




\end{document}